\documentclass{amsproc}
\usepackage{breqn,float,amssymb,pdflscape,rotating}
\setkeys{breqn}{breakdepth={1}}
\makeatletter
\@namedef{subjclassname@2020}{%
  \textup{2020} Mathematics Subject Classification}
\makeatother
\newtheorem{theorem}{Theorem}[section]

\newtheorem{proposition}[theorem]{Proposition}

\theoremstyle{definition}

\newtheorem{example}[theorem]{Example}

\theoremstyle{remark}

\numberwithin{equation}{section}
\allowdisplaybreaks
\begin{document}

\title{A Double Chebyshev Series: Derivation And Evaluation}


\author{Robert Reynolds}
\address[Robert Reynolds]{Department of Mathematics and Statistics, York University, Toronto, ON, Canada, M3J1P3}
\email[Corresponding author]{milver@my.yorku.ca}
\thanks{}

\author{ Allan Stauffer}
\address[Allan Stauffer]{Department of Mathematics and Statistics, York University, Toronto, ON, Canada, M3J1P3}
\email{stauffer@yorku.ca}
\thanks{This research is supported by NSERC Canada under Grant 504070}



\date{}

\dedicatory{}

\begin{abstract}
In this paper we use a contour integral method to derive a generating function in the form of a double series involving the product of two Chebyshev polynomials over generalized independent indices expressed in terms of the incomplete gamma function. The generating function represents a more generalized form relative to current literature. A possible application of this function to solving partial differential equations is discussed and some special cases of this generating function are derived. The work involved in the computation of this generating function is easier relative to previous methods as we have a closed form solution as opposed to numerical methods.
\end{abstract}

\maketitle
\section{Introdcution}
A double Chebyshev series is one that has two univariate Chebyshev polynomials $T_{n,p}(\alpha,\beta)=T_{n}(\alpha)T_{p}(\beta)$ as products in each term of the series, where $T_{n}(\alpha)=\cos(n\theta)$ and $\alpha=\cos(\theta)$. Literature on Chebyshev polynomials are presented in great detail in section 18 in \cite{dlmf}. These polynomials have well known generating functions in current literature and are used widely in all areas of mathematics and science. A particular popular application of these generating functions is in the solution of partial differential equations. Famous mathematicians like Lanczos \cite{lanczos} used the bivariate form  of these generating functions along with their strong convergence properties in solving ordinary differential equations. In the work by Mason \cite{mason} the bivariate form of Chebyshev polynomials was employed in studying polynomial approximation. \\\\
In this work we provide a closed form solution to the double Chebyshev series in terms of the incomplete gamma function. This in our humble opinion is perhaps a much better approach relative to previously published work. Previously published work involving the evaluation of a double Chebyshev polynomial featured approximations and power series expansions. In section 2, the Chebyshev contour integral formula is derived. In section 3, we give a detailed account of the incomplete gamma including integral and summation definitions. In section 4, the contour integral representations for the incomplete gamma function are derived. section 5, we formulate the main theorem and deduce a few propositions and examples in terms of constant and composite functions. In section 6, we look at the limiting case of the difference of the bivariate Chebyshev polynomials expressed in terms of the incomplete gamma function.\\\\

Our preliminaries start with the contour integral method \cite{reyn4}, applied to the formula for the Chebyshev generating function given by equation (18.12.8) in \cite{dlmf}. Let $a$, $\alpha$, $\beta$, $k$ and $w$ be general complex numbers, $n\in[0,\infty)$ and $p\in[0,\infty)$, where the contour integral form of the Chebyshev generating function is given by
\begin{multline}\label{eq:cheby_contour}
\frac{1}{2\pi i}\int_{C}\sum_{n,p \geq 0}T_n(\alpha ) T_p(\beta ) a^w w^{-k+n+p-1}dw\\
=\frac{1}{2\pi i}\int_{C}\frac{a^w w^{-k-1} (1-\alpha  w) (1-\beta  w)}{\left(w^2-2 \alpha  w+1\right) \left(w^2-2 \beta  w+1\right)}dw
\end{multline}
where $|Re(w)|<1$. We will use equation (\ref{eq:cheby_contour}) to derive equivalent sums for the left-hand side and a Special function form for the right-hand side.
The derivation of the summation follows the method used by us in \cite{reyn4} which involves Cauchy's integral formula. The generalized Cauchy's integral formula is given by

\begin{equation}\label{intro:cauchy}
\frac{y^k}{\Gamma(k+1)}=\frac{1}{2\pi i}\int_{C}\frac{e^{wy}}{w^{k+1}}dw.
\end{equation}
where $C$ is in general, an open contour in the complex plane where the bilinear concomitant has the same value at the end points of the contour. This method involves using a form of equation (\ref{intro:cauchy}) then multiply both sides by a function, then take a definite sum of both sides. This yields a definite sum in terms of a contour integral. A second contour integral is derived by multiplying equation (\ref{intro:cauchy}) by a function and performing some substitutions so that the contour integrals are the same.
\section{The Left-Hand Side Contour Integral}
In this section we derive the infinite sum representation involving the product of two generalized Chebyshev polynomials over independent indices for the left-hand side of equation (\ref{eq:cheby_contour}). Using a generalization of Cauchy's integral formula (\ref{intro:cauchy}), first replace $y \to \log (a),k\to k-n-p$ then multiply both sides by $T_n(\alpha ) T_p(\beta )$ and take the sums over $n\in[0,\infty)$ and $p\in [0,\infty)$ and simplify to get
\begin{multline}\label{eq:lhsci}
\sum_{n,p\geq 0}\frac{T_n(\alpha ) T_p(\beta ) \log ^{k-n-p}(a)}{\Gamma(k-n-p+1)}\\
=\frac{1}{2\pi i}\sum_{n,p\geq 0}\int_{C}a^w T_n(\alpha ) T_p(\beta ) w^{-k+n+p-1}dw\\
=\frac{1}{2\pi i}\int_{C}\sum_{n,p\geq 0}a^w T_n(\alpha ) T_p(\beta ) w^{-k+n+p-1}dw\\
=\frac{1}{2\pi i}\int_{C}\frac{a^w w^{-k-1} (1-\alpha  w) (1-\beta  w)}{\left(w^2-2 \alpha  w+1\right) \left(w^2-2 \beta  w+1\right)}dw
\end{multline}
from equation (\ref{eq:cheby_contour}) where $|Re(w)|<1$ and $Im(w)>0$ in order for the sums to converge. Apply Tonelli's theorem for multiple sums, see page 177 in \cite{gelca} as the summands are of bounded measure over the space $\mathbb{C} \times [0,\infty) \times [0,\infty) $.
\section{The Incomplete Gamma~Function}

The incomplete gamma functions~\cite{dlmf}, $\gamma(a,z)$ and $\Gamma(a,z)$, are defined by
\begin{equation}
\gamma(a,z)=\int_{0}^{z}t^{a-1}e^{-t}dt
\end{equation}
and
\begin{equation}
\Gamma(a,z)=\int_{z}^{\infty}t^{a-1}e^{-t}dt
\end{equation}
where $Re(a)>0$. The~incomplete gamma function has a recurrence relation given by
\begin{equation}
\gamma(a,z)+\Gamma(a,z)=\Gamma(a)
\end{equation}
where $a\neq 0,-1,-2,..$. The~incomplete gamma function is continued analytically by
\begin{equation}
\gamma(a,ze^{2m\pi i})=e^{2\pi mia}\gamma(a,z)
\end{equation}
and
\begin{equation}\label{eq:7}
\Gamma(a,ze^{2m\pi i})=e^{2\pi mia}\Gamma(a,z)+(1-e^{2\pi m i a})\Gamma(a)
\end{equation}
where $m\in\mathbb{Z}$, $\gamma^{*}(a,z)=\frac{z^{-a}}{\Gamma(a)}\gamma(a,z)$ is entire in $z$ and $a$. When $z\neq 0$, $\Gamma(a,z)$ is an entire function of $a$ and $\gamma(a,z)$ is meromorphic with simple poles at $a=-n$ for $n=0,1,2,...$ with residue $\frac{(-1)^n}{n!}$. These definitions are listed in Section~8.2(i) and (ii) in~\cite{dlmf}.
\section{Derivation Of The Incomplete Gamma Function Contour Integral Representations}
In this section we derive the general case of the Incomplete Gamma function in terms of the Cauchy contour integral. This formula will be used in the proceeding section to derive the equivalent Incomplete Gamma function contour integral representations for the right-hand side of equation (\ref{eq:cheby_contour}). Using a generalization of Cauchy's integral formula (\ref{intro:cauchy}), first replace $y \to y+\log (a)$ then multiply both sides by $e^{xy}$ and take the definite integral over $y\in[0,\infty)$ and simplify to get
\begin{equation}\label{eq:igf}
\frac{a^{-x} (-x)^{-k-1} \Gamma (k+1,-x \log (a))}{\Gamma(k+1)}=-\frac{1}{2\pi i}\int_{C}\frac{a^w
   w^{-k-1}}{w+x}dw
\end{equation}
from equation (3.462.12) in \cite{grad} where $Re(x)<0, \left|\arg \log(a) \right|< \pi$.
\subsection{Derivation Of The Right-Hand Side Contour Integral Representations}
In this sub-section we derive the Incomplete Gamma function in terms of the contour integral representation for the right-hand side of equation (\ref{eq:cheby_contour}). These formulae are achieved by analyzing the right-hand side of equation (\ref{eq:cheby_contour}) and decomposing the quotient of polynomials into partial fractions and applying a definite integral.
\subsubsection{Right-Hand Side First Contour Integral}
Use equation (\ref{eq:igf}) and replace $x \to -\alpha -i \sqrt{1-\alpha ^2},k\to k+1$ and multiply both sides by $\frac{i}{4 \sqrt{1-\alpha ^2} (\alpha -\beta )}$ and simplify to get
\begin{multline}\label{eq:rhs1}
\frac{i a^{\alpha +i \sqrt{1-\alpha ^2}} \left(\alpha +i \sqrt{1-\alpha ^2}\right)^{-k-2} \Gamma \left(k+2,-\left(\left(-\alpha -i \sqrt{1-\alpha ^2}\right) \log
   (a)\right)\right)}{4 \sqrt{1-\alpha ^2} \Gamma(k+2) (\alpha -\beta )}\\
   =-\frac{1}{2\pi i}\int_{C}\frac{i a^w w^{-k-2}}{4 \sqrt{1-\alpha ^2} (\alpha -\beta ) \left(-i \sqrt{1-\alpha ^2}-\alpha +w\right)}dw
\end{multline}
\subsubsection{Right-Hand Side Second Contour Integral}
Use equation (\ref{eq:rhs1}) and replace $k\to k-1$ and multiply both sides by $-\frac{i (\alpha +\beta )}{4 \sqrt{1-\alpha ^2} (\alpha -\beta )}$ and simplify to get
\begin{multline}\label{eq:rhs2}
-\frac{i t a^{\alpha +i \sqrt{1-\alpha ^2}} \left(\alpha +i \sqrt{1-\alpha ^2}\right)^{-k-1} \Gamma \left(k+1,-\left(\left(-\alpha -i \sqrt{1-\alpha ^2}\right) \log
   (a)\right)\right)}{4 \sqrt{1-\alpha ^2} \Gamma(k+1) (\alpha -\beta )}\\
   =\frac{1}{2\pi i}\int_{C}\frac{i t a^w w^{-k-1}}{4 \sqrt{1-\alpha ^2} (\alpha -\beta ) \left(-i \sqrt{1-\alpha ^2}-\alpha +w\right)}dw
\end{multline}
\subsubsection{Right-Hand Side Third Contour Integral}
Use equation (\ref{eq:igf}) and replace $x \to -\alpha +i \sqrt{1-\alpha ^2}, k\to k+1$ and multiply both sides by $-\frac{i}{4 \sqrt{1-\alpha ^2} (\alpha -\beta )}$ and simplify to get
\begin{multline}\label{eq:rhs3}
-\frac{i a^{\alpha -i \sqrt{1-\alpha ^2}} \left(\alpha -i \sqrt{1-\alpha ^2}\right)^{-k-2} \Gamma \left(k+2,-\left(\left(i \sqrt{1-\alpha ^2}-\alpha \right) \log
   (a)\right)\right)}{4 \sqrt{1-\alpha ^2} \Gamma(k+2) (\alpha -\beta )}\\
   =\frac{1}{2\pi i}\int_{C}\frac{i a^w w^{-k-2}}{4 \sqrt{1-\alpha ^2} (\alpha -\beta ) \left(i \sqrt{1-\alpha ^2}-\alpha +w\right)}dw
\end{multline}
\subsubsection{Right-Hand Side Fourth Contour Integral}
Use equation (\ref{eq:igf}) and replace $x \to -\alpha +i \sqrt{1-\alpha ^2}$ and multiply both sides by $-\frac{i (\alpha +\beta )}{4 \sqrt{1-\alpha ^2} (\alpha -\beta )}$ and simplify to get
\begin{multline}\label{eq:rhs4}
\frac{i a^{\alpha -i \sqrt{1-\alpha ^2}} (\alpha +\beta ) \left(\alpha -i \sqrt{1-\alpha ^2}\right)^{-k-1} \Gamma \left(k+1,-\left(\left(i \sqrt{1-\alpha ^2}-\alpha \right) \log
   (a)\right)\right)}{4 \sqrt{1-\alpha ^2} \Gamma(k+1) (\alpha -\beta )}\\
   =-\frac{1}{2\pi i}\int_{C}\frac{i (\alpha +\beta ) a^w w^{-k-1}}{4 \sqrt{1-\alpha ^2} (\alpha -\beta ) \left(i \sqrt{1-\alpha ^2}-\alpha
   +w\right)}dw
\end{multline}
\subsubsection{Right-Hand Side Fifth Contour Integral}
Use equation (\ref{eq:igf}) and replace $x \to -\alpha -i \sqrt{1-\alpha ^2}, k \to k-1$ and multiply both sides by $\frac{i \alpha  \beta }{4 \sqrt{1-\alpha ^2} (\alpha -\beta )}$ and simplify to get
\begin{multline}\label{eq:rhs5}
\frac{i \alpha  \beta  a^{\alpha +i \sqrt{1-\alpha ^2}} \left(\alpha +i \sqrt{1-\alpha ^2}\right)^{-k} \Gamma \left(k,-\left(\left(-\alpha -i \sqrt{1-\alpha ^2}\right) \log
   (a)\right)\right)}{4 \sqrt{1-\alpha ^2} \Gamma(k) (\alpha -\beta )}\\
   =-\frac{1}{2\pi i}\int_{C}\frac{i \alpha  \beta  a^w w^{-k}}{4 \sqrt{1-\alpha ^2} (\alpha -\beta ) \left(-i \sqrt{1-\alpha ^2}-\alpha
   +w\right)}dw
\end{multline}
\subsubsection{Right-Hand Side Sixth Contour Integral}
Use equation (\ref{eq:igf}) and replace $x \to -\alpha +i \sqrt{1-\alpha ^2}, k \to k-1$ and multiply both sides by $-\frac{i \alpha  \beta }{4 \sqrt{1-\alpha ^2} (\alpha -\beta )}$ and simplify to get
\begin{multline}\label{eq:rhs6}
-\frac{i \alpha  \beta  a^{\alpha -i \sqrt{1-\alpha ^2}} \left(\alpha -i \sqrt{1-\alpha ^2}\right)^{-k} \Gamma \left(k,-\left(\left(i \sqrt{1-\alpha ^2}-\alpha \right) \log
   (a)\right)\right)}{4 \sqrt{1-\alpha ^2} \Gamma(k) (\alpha -\beta )}\\
   =\frac{1}{2\pi i}\int_{C}\frac{i \alpha  \beta  a^w w^{-k}}{4 \sqrt{1-\alpha ^2} (\alpha -\beta ) \left(i \sqrt{1-\alpha ^2}-\alpha
   +w\right)}dw
\end{multline}
\subsubsection{Right-Hand Side Seventh Contour Integral}
Use equation (\ref{eq:igf}) and replace $x \to -\beta -i \sqrt{1-\beta ^2}, k \to k+1$ and multiply both sides by $-\frac{i}{4 \sqrt{1-\beta ^2} (\alpha -\beta )}$ and simplify to get
\begin{multline}\label{eq:rhs7}
-\frac{i a^{\beta +i \sqrt{1-\beta ^2}} \left(\beta +i \sqrt{1-\beta ^2}\right)^{-k-2} \Gamma \left(k+2,-\left(\left(-\beta -i \sqrt{1-\beta ^2}\right) \log (a)\right)\right)}{4
   \sqrt{1-\beta ^2} \Gamma(k+2) (\alpha -\beta )}\\
   =\frac{1}{2\pi i}\int_{C}\frac{i a^w w^{-k-2}}{4 \sqrt{1-\beta ^2} (\alpha -\beta ) \left(-i \sqrt{1-\beta ^2}-\beta +w\right)}dw
\end{multline}
\subsubsection{Right-Hand Side Eighth Contour Integral}
Use equation (\ref{eq:igf}) and replace $x \to -\beta -i \sqrt{1-\beta ^2}$ and multiply both sides by $\frac{i (\alpha +\beta )}{4 \sqrt{1-\beta ^2} (\alpha -\beta )}$ and simplify to get
\begin{multline}\label{eq:rhs8}
\frac{i t a^{\beta +i \sqrt{1-\beta ^2}} \left(\beta +i \sqrt{1-\beta ^2}\right)^{-k-1} \Gamma \left(k+1,-\left(\left(-\beta -i \sqrt{1-\beta ^2}\right) \log (a)\right)\right)}{4
   \sqrt{1-\beta ^2} \Gamma(k+1) (\alpha -\beta )}\\
   =-\frac{1}{2\pi i}\int_{C}\frac{i t a^w w^{-k-1}}{4 \sqrt{1-\beta ^2} (\alpha -\beta ) \left(-i \sqrt{1-\beta ^2}-\beta +w\right)}dw
\end{multline}
\subsubsection{Right-Hand Side Ninth Contour Integral}
Use equation (\ref{eq:igf}) and replace $x \to -\beta -i \sqrt{1-\beta ^2}, k \to k-1$ and multiply both sides by $\frac{i \alpha  \beta }{4 \sqrt{1-\beta ^2} (\alpha -\beta )}$ and simplify to get
\begin{multline}\label{eq:rhs9}
-\frac{i \alpha  \beta  a^{\beta +i \sqrt{1-\beta ^2}} \left(\beta +i \sqrt{1-\beta ^2}\right)^{-k} \Gamma \left(k,-\left(\left(-\beta -i \sqrt{1-\beta ^2}\right) \log
   (a)\right)\right)}{4 \sqrt{1-\beta ^2} \Gamma(k) (\alpha -\beta )}\\
   =\frac{1}{2\pi i}\int_{C}\frac{i \alpha  \beta  a^w w^{-k}}{4 \sqrt{1-\beta ^2} (\alpha -\beta ) \left(-i \sqrt{1-\beta ^2}-\beta
   +w\right)}dw
\end{multline}
\subsubsection{Right-Hand Side Tenth Contour Integral}
Use equation (\ref{eq:igf}) and replace $x \to -\beta +i \sqrt{1-\beta ^2}, k \to k+1$ and multiply both sides by $\frac{i}{4 \sqrt{1-\beta ^2} (\alpha -\beta )}$ and simplify to get
\begin{multline}\label{eq:rhs10}
\frac{i a^{\beta -i \sqrt{1-\beta ^2}} \left(\beta -i \sqrt{1-\beta ^2}\right)^{-k-2} \Gamma \left(k+2,-\left(\left(i \sqrt{1-\beta ^2}-\beta \right) \log (a)\right)\right)}{4
   \sqrt{1-\beta ^2} \Gamma(k+2) (\alpha -\beta )}\\
   =-\frac{1}{2\pi i}\int_{C}\frac{i a^w w^{-k-2}}{4 \sqrt{1-\beta ^2} (\alpha -\beta ) \left(i \sqrt{1-\beta ^2}-\beta +w\right)}dw
\end{multline}
\subsubsection{Right-Hand Side Eleventh Contour Integral}
Use equation (\ref{eq:igf}) and replace $x \to -\beta +i \sqrt{1-\beta ^2}$ and multiply both sides by $-\frac{i (\alpha +\beta )}{4 \sqrt{1-\beta ^2} (\alpha -\beta )}$ and simplify to get
\begin{multline}\label{eq:rhs11}
-\frac{i a^{\beta -i \sqrt{1-\beta ^2}} (\alpha +\beta ) \left(\beta -i \sqrt{1-\beta ^2}\right)^{-k-1} \Gamma \left(k+1,-\left(\left(i \sqrt{1-\beta ^2}-\beta \right) \log
   (a)\right)\right)}{4 \sqrt{1-\beta ^2} \Gamma(k+1) (\alpha -\beta )}\\
   =\frac{1}{2\pi i}\int_{C}\frac{i (\alpha +\beta ) a^w w^{-k-1}}{4 \sqrt{1-\beta ^2} (\alpha -\beta ) \left(i \sqrt{1-\beta ^2}-\beta
   +w\right)}dw
\end{multline}
\subsubsection{Right-Hand Side Twelfth Contour Integral}
Use equation (\ref{eq:igf}) and replace $x \to -\beta +i \sqrt{1-\beta ^2}, k\to k-1$ and multiply both sides by $\frac{i \alpha  \beta }{4 \sqrt{1-\beta ^2} (\alpha -\beta )}$ and simplify to get
\begin{multline}\label{eq:rhs12}
\frac{i \alpha  \beta  a^{\beta -i \sqrt{1-\beta ^2}} \left(\beta -i \sqrt{1-\beta ^2}\right)^{-k} \Gamma \left(k,-\left(\left(i \sqrt{1-\beta ^2}-\beta \right) \log
   (a)\right)\right)}{4 \sqrt{1-\beta ^2} \Gamma(k) (\alpha -\beta )}\\
   =-\frac{1}{2\pi i}\int_{C}\frac{i \alpha  \beta  a^w w^{-k}}{4 \sqrt{1-\beta ^2} (\alpha -\beta ) \left(i \sqrt{1-\beta ^2}-\beta
   +w\right)}dw
\end{multline}
\section{Main Results}
In this section we develop a theorem, propositions and  examples to demonstrate the potential application of the double Chebyshev series. 
\begin{theorem}
For all $a,k,\alpha,\beta\in\mathbb{C}$ then,
\begin{multline}\label{eq:cheby}
\sum_{n,p\geq 0}\frac{ T_n(\alpha ) T_p(\beta )}{(a\pi)^{n+p} (k)_{-n-p+1}}
=\frac{1  }{4 i k (k+1) (a\pi)^{k}\sqrt{1-\alpha ^2} (\alpha -\beta ) \sqrt{1-\beta
   ^2}}\\
    \left(e^{a \pi  \left(\alpha -i \sqrt{1-\alpha ^2}\right)} \sqrt{1-\beta ^2} \Gamma
   \left(k+2,a \pi  \left(\alpha -i \sqrt{1-\alpha ^2}\right)\right) \left(\alpha -i \sqrt{1-\alpha ^2}\right)^{-k-2}\right.\\
   \left.-e^{a \pi  \left(\alpha -i \sqrt{1-\alpha ^2}\right)} (k+1) \alpha 
   \sqrt{1-\beta ^2} \Gamma \left(k+1,a \pi  \left(\alpha -i \sqrt{1-\alpha ^2}\right)\right) \left(\alpha -i \sqrt{1-\alpha ^2}\right)^{-k-1}\right.\\
   \left.-e^{a \pi  \left(\alpha -i \sqrt{1-\alpha
   ^2}\right)} (k+1) \beta  \sqrt{1-\beta ^2} \Gamma \left(k+1,a \pi  \left(\alpha -i \sqrt{1-\alpha ^2}\right)\right) \left(\alpha -i \sqrt{1-\alpha ^2}\right)^{-k-1}\right.\\
   \left.+e^{a \pi 
   \left(\alpha -i \sqrt{1-\alpha ^2}\right)} k (k+1) \alpha  \beta  \sqrt{1-\beta ^2} \Gamma \left(k,a \pi  \left(\alpha -i \sqrt{1-\alpha ^2}\right)\right) \left(\alpha -i \sqrt{1-\alpha
   ^2}\right)^{-k}\right.\\
   \left.-e^{a \pi  \left(\alpha +i \sqrt{1-\alpha ^2}\right)} k (k+1) \alpha  \left(\alpha +i \sqrt{1-\alpha ^2}\right)^{-k} \beta  \sqrt{1-\beta ^2} \Gamma \left(k,a \pi 
   \left(\alpha +i \sqrt{1-\alpha ^2}\right)\right)\right.\\
   \left.-e^{a \pi  \left(\beta -i \sqrt{1-\beta ^2}\right)} k (k+1) \alpha  \sqrt{1-\alpha ^2} \beta  \left(\beta -i \sqrt{1-\beta
   ^2}\right)^{-k} \Gamma \left(k,a \pi  \left(\beta -i \sqrt{1-\beta ^2}\right)\right)\right.\\
   \left.+e^{a \pi  \left(\beta +i \sqrt{1-\beta ^2}\right)} k (k+1) \alpha  \sqrt{1-\alpha ^2} \beta 
   \left(\beta +i \sqrt{1-\beta ^2}\right)^{-k} \Gamma \left(k,a \pi  \left(\beta +i \sqrt{1-\beta ^2}\right)\right)\right.\\
   \left.+e^{a \pi  \left(\alpha +i \sqrt{1-\alpha ^2}\right)} (k+1) \alpha 
   \left(\alpha +i \sqrt{1-\alpha ^2}\right)^{-k-1} \sqrt{1-\beta ^2} \Gamma \left(k+1,a \pi  \left(\alpha +i \sqrt{1-\alpha ^2}\right)\right)\right.\\
   \left.+e^{a \pi  \left(\alpha +i \sqrt{1-\alpha
   ^2}\right)} (k+1) \left(\alpha +i \sqrt{1-\alpha ^2}\right)^{-k-1} \beta  \sqrt{1-\beta ^2} \Gamma \left(k+1,a \pi  \left(\alpha +i \sqrt{1-\alpha ^2}\right)\right)\right.\\
   \left.+e^{a \pi 
   \left(\beta -i \sqrt{1-\beta ^2}\right)} (k+1) \sqrt{1-\alpha ^2} \beta  \left(\beta -i \sqrt{1-\beta ^2}\right)^{-k-1} \Gamma \left(k+1,a \pi  \left(\beta -i \sqrt{1-\beta
   ^2}\right)\right)\right.\\
   \left.+e^{a \pi  \left(\beta -i \sqrt{1-\beta ^2}\right)} (k+1) \alpha  \sqrt{1-\alpha ^2} \left(\beta -i \sqrt{1-\beta ^2}\right)^{-k-1} \Gamma \left(k+1,a \pi  \left(\beta
   -i \sqrt{1-\beta ^2}\right)\right)\right.\\
   \left.-e^{a \pi  \left(\beta +i \sqrt{1-\beta ^2}\right)} (k+1) \sqrt{1-\alpha ^2} \beta  \left(\beta +i \sqrt{1-\beta ^2}\right)^{-k-1} \Gamma \left(k+1,a
   \pi  \left(\beta +i \sqrt{1-\beta ^2}\right)\right)\right.\\
   \left.-e^{a \pi  \left(\beta +i \sqrt{1-\beta ^2}\right)} (k+1) \alpha  \sqrt{1-\alpha ^2} \left(\beta +i \sqrt{1-\beta ^2}\right)^{-k-1}
   \Gamma \left(k+1,a \pi  \left(\beta +i \sqrt{1-\beta ^2}\right)\right)\right.\\
   \left.-e^{a \pi  \left(\alpha +i \sqrt{1-\alpha ^2}\right)} \left(\alpha +i \sqrt{1-\alpha ^2}\right)^{-k-2}
   \sqrt{1-\beta ^2} \Gamma \left(k+2,a \pi  \left(\alpha +i \sqrt{1-\alpha ^2}\right)\right)\right.\\
   \left.-e^{a \pi  \left(\beta -i \sqrt{1-\beta ^2}\right)} \sqrt{1-\alpha ^2} \left(\beta -i
   \sqrt{1-\beta ^2}\right)^{-k-2} \Gamma \left(k+2,a \pi  \left(\beta -i \sqrt{1-\beta ^2}\right)\right)\right.\\
   \left.+e^{a \pi  \left(\beta +i \sqrt{1-\beta ^2}\right)} \sqrt{1-\alpha ^2} \left(\beta
   +i \sqrt{1-\beta ^2}\right)^{-k-2} \Gamma \left(k+2,a \pi  \left(\beta +i \sqrt{1-\beta ^2}\right)\right)\right)
\end{multline}
\end{theorem}
\begin{proof}
Observe the right-hand side of equation (\ref{eq:lhsci}) is equivalent to the addition of the right-hand sides of equations (\ref{eq:rhs1}) to (\ref{eq:rhs12}) then the left-hand sides are equal. We then apply equation (\ref{eq:cheby_contour}) and replace $a\to e^{a\pi}$ and simplify the Gamma function and use the Pochhammer symbol see equation (5.2.5) in \cite{dlmf} to yield the stated result.
\end{proof}
\begin{proposition}
\begin{equation}
\sum_{n,p \geq 0}\frac{1}{(a \pi )^{n+p} (k)_{1-n-p}}=1+\frac{1}{k}+e^{a \pi } (1+k-a \pi ) E_{1-k}(a \pi )
\end{equation}
\end{proposition}
\begin{proof}
Use equation (\ref{eq:cheby}) and apply l'Hopital's rule as $\alpha \to 1, \beta \to 1$ and simplify using equations (18.5.14) and (8.19.1) in \cite{dlmf}.
\end{proof}
\begin{proposition}
\begin{multline}\label{eq:prod_cos}
\sum_{n,p \geq 0}\frac{\cos (n \alpha ) \cos (p \beta )}{(a \pi )^{n+p} (k)_{1-n-p}}\\
=\frac{1}{4 k i (1+k) (a \pi
   )^k (\cos (\alpha )-\cos (\beta ))}\left(e^{a e^{-i \alpha } \pi }
   \left(e^{-i \alpha }\right)^{-k} k (1+k) \cos (\beta ) \cot (\alpha ) \Gamma \left(k,a e^{-i \alpha } \pi
   \right)\right.\\
   \left.-e^{a e^{i \alpha } \pi } \left(e^{i \alpha }\right)^{-k} k (1+k) \cos (\beta ) \cot (\alpha ) \Gamma
   \left(k,a e^{i \alpha } \pi \right)\right.\\
   \left.-k (1+k) \cos (\alpha ) \cot (\beta ) \left(e^{a e^{-i \beta } \pi } \left(e^{i
   \beta }\right)^k \Gamma \left(k,a e^{-i \beta } \pi \right)-e^{a e^{i \beta } \pi } \left(e^{-i \beta }\right)^k
   \Gamma \left(k,a e^{i \beta } \pi \right)\right)\right.\\
   \left.-e^{a e^{-i \alpha } \pi +i \alpha } \left(e^{-i \alpha
   }\right)^{-k} (1+k) (\cos (\alpha )+\cos (\beta )) \csc (\alpha ) \Gamma \left(1+k,a e^{-i \alpha } \pi
   \right)\right.\\
   \left.+e^{a e^{i \alpha } \pi -i \alpha } \left(e^{i \alpha }\right)^{-k} (1+k) (\cos (\alpha )+\cos (\beta ))
   \csc (\alpha ) \Gamma \left(1+k,a e^{i \alpha } \pi \right)\right.\\
   \left.+e^{a e^{-i \beta } \pi +i \beta } \left(e^{i \beta
   }\right)^k (1+k) (\cos (\alpha )+\cos (\beta )) \csc (\beta ) \Gamma \left(1+k,a e^{-i \beta } \pi \right)\right.\\
   \left.-e^{a
   e^{i \beta } \pi -i \beta } \left(e^{-i \beta }\right)^k (1+k) (\cos (\alpha )+\cos (\beta )) \csc (\beta ) \Gamma
   \left(1+k,a e^{i \beta } \pi \right)\right.\\
   \left.+e^{a e^{-i \alpha } \pi +2 i \alpha } \left(e^{-i \alpha }\right)^{-k} \csc
   (\alpha ) \Gamma \left(2+k,a e^{-i \alpha } \pi \right)-e^{a e^{i \alpha } \pi -2 i \alpha } \left(e^{i \alpha
   }\right)^{-k} \csc (\alpha ) \Gamma \left(2+k,a e^{i \alpha } \pi \right)\right.\\
   \left.-e^{a e^{-i \beta } \pi +2 i \beta }
   \left(e^{i \beta }\right)^k \csc (\beta ) \Gamma \left(2+k,a e^{-i \beta } \pi \right)+e^{a e^{i \beta } \pi -2 i
   \beta } \left(e^{-i \beta }\right)^k \csc (\beta ) \Gamma \left(2+k,a e^{i \beta } \pi \right)\right)
\end{multline}
\end{proposition}
\begin{proof}
Use equation (\ref{eq:cheby}) replace $\alpha \to \cos (\alpha ),\beta \to \cos (\beta )$ and simplify the left-hand side using equation (18.5.1) in \cite{dlmf}.
\end{proof}
\begin{example}
Product of Cosine Functions in terms of the Complementary Error Function $\textup{erfc}(z)$ and error function $\textup{erf}(z)$.
\begin{multline}
\sum_{n,p \geq 0}\frac{e^{-4 (n+p)} \cos \left(\frac{\pi  n}{2}\right) \cos \left(\frac{\pi 
   p}{4}\right)}{\left(-\frac{1}{2}\right)_{-n-p+1}}\\
   =\left(\frac{1}{4}+\frac{i}{4}\right) e^2 \sqrt{\pi }
   \left(\left(2+i \sqrt{2}\right) e^{-i e^4} \left(\text{erf}\left((-1)^{3/4} e^2\right)+1\right)\right.\\
   \left.+2 (-1)^{7/8}
   e^{-(-1)^{3/4} e^4} \left(\text{erf}\left((-1)^{7/8} e^2\right)+1\right)\right.\\
   \left.+2 (-1)^{5/8} e^{\sqrt[4]{-1} e^4}
   \text{erfc}\left(\sqrt[8]{-1} e^2\right)-\left(\sqrt{2}+2 i\right) e^{i e^4} \text{erfc}\left(\sqrt[4]{-1}
   e^2\right)\right)
\end{multline}
\end{example}
\begin{proof}
Use equation (\ref{eq:prod_cos}) set $k=-1/2,a=e^4/\pi,\alpha =\pi/2,\beta =\pi/4$ and simplify the right-hand side using equations (8.4.1) and (8.4.6) in \cite{dlmf}.
\end{proof}
\begin{example}
The Golden Ratio.
\begin{multline}
\sum_{n,p \geq 0}\frac{T_n\left(\sqrt{5}\right) T_p\left(\frac{\sqrt{5}}{2}\right)}{(a \pi )^{n+p}
   (k)_{1-n-p}}\\
   =\frac{1}{k}+\frac{1}{4 \sqrt{5}}\left(\left(4+\sqrt{5}\right) e^{\left(-2+\sqrt{5}\right) a \pi }
   E_{1-k}\left(\left(-2+\sqrt{5}\right) a \pi \right)\right.\\\left.+\left(-1+\sqrt{5}\right) e^{\frac{1}{2}
   \left(-1+\sqrt{5}\right) a \pi } E_{1-k}\left(\frac{1}{2} \left(-1+\sqrt{5}\right) a \pi
   \right)\right.\\\left.+\left(1+\sqrt{5}\right) e^{\frac{1}{2} \left(1+\sqrt{5}\right) a \pi } E_{1-k}\left(\frac{1}{2}
   \left(1+\sqrt{5}\right) a \pi \right)\right.\\\left.+\left(-4+\sqrt{5}\right) e^{\left(2+\sqrt{5}\right) a \pi }
   E_{1-k}\left(\left(2+\sqrt{5}\right) a \pi \right)\right)
\end{multline}
\end{example}
\begin{proof}
Use equation (\ref{eq:prod_cos}) set $\alpha \to \sqrt{5},\beta \to \frac{\sqrt{5}}{2}$ and simplify the right-hand side using equations (8.4.1) and (8.4.6) in \cite{dlmf}.
\end{proof}
\section{The Limiting Case Of The Difference Of A Negative Index}
In this section we will derive a few generating functions using the identity $T_{n}(-x)=-1^nT_{n}(x)$ which is listed in Table (18.6.1) in \cite{dlmf}. We proceed by using equation (\ref{eq:cheby}) and forming a second equation by replacing $\alpha \to \alpha, \beta\to -\beta$ taking their difference and simplifying. Next we evaluate five cases listed below when $\alpha=\beta=1,\alpha=\beta=2, \alpha=\beta=3, \alpha=\beta=4, \alpha=\beta=5$. In order to simplify the right-hand sides of these formulae we apply the limit and simplify. The simplification process is not very easy and tedious, however the results are very inetersting.
\begin{example}
\begin{multline}
\sum_{n,p \geq 0}\frac{-1+(-1)^{n+p}}{(a \pi )^{n+p} (k)_{1-n-p}}=\frac{e^{-(a+i k) \pi }}{(a \pi )^k k} \left(k (1+k+a \pi ) \Gamma (k,-a \pi
   )\right.\\
   \left.+e^{a \pi } \left(\left((-a)^k-a^k e^{i k \pi }\right) (1+k) \pi ^k-e^{(a+i k) \pi } k (1+k-a \pi ) \Gamma (k,a \pi
   )\right)\right)
\end{multline}
\end{example}
%
%
\begin{example}
\begin{multline}
\sum_{n,p \geq 0}\frac{\left(-1+(-1)^{n+p}\right) T_n(2) T_p(2)}{(a \pi )^{n+p} (k)_{1-n-p}}\\
=\frac{a^{-k}
   e^{-\left(\left(2+\sqrt{3}\right) a+i k\right) \pi } }{12 k}
   \left(6 e^{\left(2+\sqrt{3}\right) a \pi }
   \left((-a)^k-a^k e^{i k \pi }\right) (2+k)\right.\\\left.+k \left(-a^k e^{4 a \pi +i k \pi } \left(6+2 \sqrt{3}+3 k+3
   \left(-2+\sqrt{3}\right) a \pi \right) E_{1-k}\left(-\left(\left(-2+\sqrt{3}\right) a \pi \right)\right)\right.\right.\\
   \left.\left.+(-a)^k
   e^{2 \sqrt{3} a \pi } \left(6+2 \sqrt{3}+3 k-3 \left(-2+\sqrt{3}\right) a \pi \right)
   E_{1-k}\left(\left(-2+\sqrt{3}\right) a \pi \right)\right.\right.\\ \left.\left.+(-a)^k \left(6-2 \sqrt{3}+3 k+3 \left(2+\sqrt{3}\right) a \pi
   \right) E_{1-k}\left(-\left(\left(2+\sqrt{3}\right) a \pi \right)\right)\right.\right.\\ \left.\left.+a^k e^{\left(2 \left(2+\sqrt{3}\right) a+i
   k\right) \pi } \left(-6+2 \sqrt{3}-3 k+3 \left(2+\sqrt{3}\right) a \pi \right) E_{1-k}\left(\left(2+\sqrt{3}\right)
   a \pi \right)\right)\right)
\end{multline}
\end{example}
%
%
\begin{example}
\begin{multline}
\sum_{n,p \geq 0}\frac{\left(-1+(-1)^{n+p}\right) T_n(3) T_p(3)}{(a \pi )^{n+p} (k)_{1-n-p}}\\
=\frac{1}{16 k}\left(-3+2
   \sqrt{2}\right)^{-2 k} a^{-k} \left(\left(-3+2 \sqrt{2}\right) a\right)^{-k}
    e^{\left(-\left(\left(3+2
   \sqrt{2}\right) a\right)+2 i k\right) \pi }\\
    \left(-\left(\left(3+2 \sqrt{2}\right) \pi \right)\right)^{-k} \left(8
   \left(3-2 \sqrt{2}\right)^{2 k} (-a)^k e^{\left(3+2 \sqrt{2}\right) a \pi } \left((-a)^k-a^k e^{i k \pi }\right)
   (2+k) \pi ^k\right.\\ \left.+\left(3-2 \sqrt{2}\right)^{3 k} (-a)^k k \left(8-3 \sqrt{2}+4 k+4 \left(3+2 \sqrt{2}\right) a \pi
   \right) \Gamma \left(k,-\left(\left(3+2 \sqrt{2}\right) a \pi \right)\right)\right.\\ \left.+\left(\left(-3+2 \sqrt{2}\right)
   a\right)^k k \left(-e^{6 a \pi +i k \pi } \left(8+3 \sqrt{2}+4 k+4 \left(-3+2 \sqrt{2}\right) a \pi \right) \Gamma
   \left(k,\left(3-2 \sqrt{2}\right) a \pi \right)\right.\right.\\ \left.\left.+e^{4 \sqrt{2} a \pi } \left(\left(8+3 \sqrt{2}+4 k+4 \left(3-2
   \sqrt{2}\right) a \pi \right) \Gamma \left(k,\left(-3+2 \sqrt{2}\right) a \pi \right)\right.\right.\right.\\ \left.\left.\left.+\left(3-2 \sqrt{2}\right)^{2
   k} e^{6 a \pi +i k \pi } \left(-8+3 \sqrt{2}-4 k+4 \left(3+2 \sqrt{2}\right) a \pi \right) \Gamma \left(k,\left(3+2
   \sqrt{2}\right) a \pi \right)\right)\right)\right)
\end{multline}
\end{example}
%
%
\begin{example}
\begin{multline}
\sum_{n,p \geq 0}\frac{\left(-1+(-1)^{n+p}\right) T_n(4) T_p(4)}{(a \pi )^{n+p} (k)_{1-n-p}}\\
=\frac{a^{-k}
   e^{-\left(\left(4+\sqrt{15}\right) a \pi \right)} \left(\left(-4+\sqrt{15}\right) \pi \right)^{-k}}{60 k}
    \left(30
   e^{\left(4+\sqrt{15}\right) a \pi } \left((-a)^k-a^k e^{i k \pi }\right) (2+k) \left(\left(4-\sqrt{15}\right) \pi
   \right)^k\right.\\ \left.-e^{8 a \pi +i k \pi } k \left(30+4 \sqrt{15}+15 k+15 \left(-4+\sqrt{15}\right) a \pi \right) \Gamma
   \left(k,-\left(\left(-4+\sqrt{15}\right) a \pi \right)\right)\right.\\ \left.+e^{2 \sqrt{15} a \pi } k \left(30+4 \sqrt{15}+15 k-15
   \left(-4+\sqrt{15}\right) a \pi \right) \Gamma \left(k,\left(-4+\sqrt{15}\right) a \pi \right)\right.\\ \left.+\left(31-8
   \sqrt{15}\right)^k k \left(\left(30-4 \sqrt{15}+15 k+15 \left(4+\sqrt{15}\right) a \pi \right) \Gamma
   \left(k,-\left(\left(4+\sqrt{15}\right) a \pi \right)\right)\right.\right.\\ \left.\left.+e^{\left(2 \left(4+\sqrt{15}\right) a+i k\right) \pi }
   \left(-30+4 \sqrt{15}-15 k+15 \left(4+\sqrt{15}\right) a \pi \right) \Gamma \left(k,\left(4+\sqrt{15}\right) a \pi
   \right)\right)\right)
\end{multline}
\end{example}
%
%
\begin{example}
\begin{multline}
\sum_{n,p \geq 0}\frac{\left(-1+(-1)^{n+p}\right) T_n(5) T_p(5)}{(a \pi )^{n+p} (k)_{1-n-p}}\\
=\frac{\left(-5+2
   \sqrt{6}\right)^{-2 k} a^{-k} \left(\left(-5+2 \sqrt{6}\right) a\right)^{-k} }{48 k}
    e^{\left(-\left(\left(5+2
   \sqrt{6}\right) a\right)+2 i k\right) \pi } \left(-\left(\left(5+2 \sqrt{6}\right) \pi \right)\right)^{-k} \\\left(24
   \left(5-2 \sqrt{6}\right)^{2 k}
    (-a)^k e^{\left(5+2 \sqrt{6}\right) a \pi } \left((-a)^k-a^k e^{i k \pi }\right)
   (2+k) \pi ^k+\left(5-2 \sqrt{6}\right)^{3 k} \right.\\ \left.(-a)^k k \left(24-5 \sqrt{6}+12 k+12 \left(5+2 \sqrt{6}\right) a \pi
   \right) \Gamma \left(k,-\left(\left(5+2 \sqrt{6}\right) a \pi \right)\right)\right.\\ \left.+\left(\left(-5+2 \sqrt{6}\right)
   a\right)^k k \left(-e^{10 a \pi +i k \pi } \left(24+5 \sqrt{6}+12 k+12 \left(-5+2 \sqrt{6}\right) a \pi \right)\right.\right.\\ \left.\left.
   \Gamma \left(k,\left(5-2 \sqrt{6}\right) a \pi \right)+e^{4 \sqrt{6} a \pi } \left(\left(24+5 \sqrt{6}+12 k+12
   \left(5-2 \sqrt{6}\right) a \pi \right)\right.\right.\right.\\ \left.\left.\left. \Gamma \left(k,\left(-5+2 \sqrt{6}\right) a \pi \right)+\left(5-2
   \sqrt{6}\right)^{2 k} e^{10 a \pi +i k \pi } \left(-24+5 \sqrt{6}-12 k +12 \left(5+2 \sqrt{6}\right) a \pi \right)\right.\right.\right.\\ \left.\left.\left.
   \Gamma \left(k,\left(5+2 \sqrt{6}\right) a \pi \right)\right)\right)\right)
\end{multline}
\end{example}
%
%
\section{Conclusion}
In this paper, we have presented a method for deriving a generating function involving the product of two Chebyshev polynomials over independent parameter ranges along with some interesting special cases using contour integration. The results presented were numerically verified for both real and imaginary and complex values of the parameters in the integrals using Mathematica by Wolfram.
\end{document}